\documentclass[a4paper,11pt]{article}
\usepackage{amsmath}
\usepackage{amssymb}
\usepackage{theorem}
\usepackage{pstricks}
\usepackage{euscript}
\usepackage{epic,eepic}
\usepackage{graphicx}
\PassOptionsToPackage{normalem}{ulem}
\usepackage{amsmath}
\usepackage{amssymb}
\usepackage{theorem}
\usepackage{pstricks}
\usepackage{euscript}
\usepackage{epic,eepic}
\usepackage{graphicx}
\usepackage{setspace}

\usepackage{algorithm}
\usepackage{algpseudocode}
\usepackage{enumitem}
\setlist[enumerate]{itemsep=5pt}

\PassOptionsToPackage{normalem}{ulem}

\newcommand{\menge}[2]{\big\{{#1} \;|\; {#2}\big\}}

\newcommand{\emp}{\ensuremath{{\varnothing}}}

\newcommand{\scal}[2]{\left\langle{#1}\mid {#2} \right\rangle}

\newcommand{\RR}{\ensuremath{\mathbb R}}

\newcommand{\NN}{\ensuremath{\mathbb N}}

\newcommand{\dom}{\ensuremath{\operatorname{dom}}}

\newcommand{\prox}{\ensuremath{\operatorname{prox}}}

\newcommand{\E}{\ensuremath{\mathbf{E}}}

\newcommand{\Id}{\ensuremath{\operatorname{Id}}}

\newtheorem{theorem}{Theorem}[section]

\newtheorem{proposition}[theorem]{Proposition}

\theoremstyle{plain}{\theorembodyfont{\rmfamily}
}
\theoremstyle{plain}{\theorembodyfont{\rmfamily}
}
\theoremstyle{plain}{\theorembodyfont{\rmfamily}
\theoremstyle{plain}{\theorembodyfont{\rmfamily}
\newtheorem{example}[theorem]{Example}}
\theoremstyle{plain}{\theorembodyfont{\rmfamily}
\newtheorem{problem}[theorem]{Problem}}
\theoremstyle{plain}{\theorembodyfont{\rmfamily}
\newtheorem{remark}[theorem]{Remark}}
\theoremstyle{plain}{\theorembodyfont{\rmfamily}
}

\definecolor{labelkey}{rgb}{0,0.08,0.45}
\definecolor{refkey}{rgb}{0,0.6,0.0}
\definecolor{Brown}{rgb}{0.45,0.0,0.05}
\definecolor{dgreen}{rgb}{0.00,0.49,0.00}
\definecolor{dblue}{rgb}{0,0.08,0.75}
\numberwithin{equation}{section}

\begin{document}
\title{\sffamily\huge 
On the linear convergence of the  stochastic gradient method with constant step-size
}
\author{Volkan Cevher and B$\grave{\text{\u{a}}}$ng C\^ong V\~u\\[5mm]
 Laboratory for Information and Inference Systems (LIONS), \\EPFL, Switzerland\\ 
 \{volkan.cevher,bang.vu\}@epfl.ch
}
\date{}
\maketitle
\vspace{-10mm}
\begin{abstract}\vspace{-3mm}
The strong growth condition (SGC) is known to be a sufficient condition for linear convergence of the stochastic gradient method using a constant step-size $\gamma$ (SGM-CS). 
In this paper, we provide  a necessary condition, for the linear convergence of SGM-CS, that is weaker than SGC. 
Moreover, when this necessary is violated up to a additive perturbation $\sigma$, we show that  both the projected stochastic gradient method using a constant step-size (PSGM-CS) and  the proximal stochastic gradient method exhibit linear convergence to a noise dominated region, whose distance to the optimal solution is proportional to $\gamma \sigma$.
\end{abstract}




\vspace{-5mm}
\section{Introduction}\vspace{-5mm}
In this paper, we consider the following stochastic convex optimization problem, which is widely studied in the literature;
 \textit{cf.}, \cite{Duchi09,LSB14a,Com02} for instances.
\begin{problem} \label{prob1}
Let $f\colon\RR^d\to\RR$ be a convex differentiable function with $L$-Lipschitz continuous gradient with an expectation form $f(x)= \E_\xi [K(x,\xi)]$. In the expectation,  $\xi$ is a random vector whose probability distribution $P$ is supported on set a $\Omega \subset \RR^m$, and $K\colon \RR^d\times\Omega\to \RR$ is convex function with respect to the variable $x$.
Let $g\colon\RR^d\to \left]-\infty,+\infty\right]$ be a proper lower semicontinuous convex function. Based on this setup, the problem we are interested in studying can be written as 
\begin{equation}
\underset{x\in \RR^d}{\text{minimize}} \; f(x) +g(x),
\end{equation}
under the following assumptions: 
\begin{enumerate}
\item It is possible to obtain independent and identically distributed (iid) samples of $(\xi_t)_{t\in\NN}$ of $\xi$.
\item Given $(x_t,\xi_t)\in \RR^d\times\Omega$,
one can find a point $\nabla K(x_t,\xi_t)$ such that $ \E_{\xi_t}[\nabla K(x_t,\xi_t)] = \nabla f(x_t)$. Here, the gradient $\nabla K(x,\xi)$ is taken with respect to $x$.
\end{enumerate}
\end{problem}

The proximal stochastic gradient method  (\textit{cf.}, \cite{Duchi09,LSB14a,Com02}, and the references therein) is an elementary method
for solving Problem \ref{prob1}. This method is extremely simple and highly scalable since 
it only uses  the proximity operator of $g$ and an unbiased estimate of the gradient of $f$ at each iteration. 
Hence, the method is  popular in machine learning and signal processing applications.
 
In this paper, we focus our attention particularly to the case where $g$ is the indicator of some nonempty, closed convex set $C$
(\textit{cf.},  \cite{Nem08, Shamir13} and the references therein). Then, the proximal stochastic gradient method reduces to 
 the projected stochastic gradient method (PSGM):
\begin{equation}
\label{PSG}
 x_0\in C\; \text{and}\; (\forall t\in\NN)\; x_{t+1} = P_C(x_t-\gamma_t \nabla K(x_t,\xi_t)),
\end{equation}
where $\gamma_t> 0$ is the step size. When $C$ is the whole space,  \eqref{PSG} is the 
stochastic gradient method (SGM). 

While the computational cost of these stochastic methods is much cheaper than their deterministic counterparts, their slow convergence rate is problematic for obtaining high accuracy solutions. Indeed, even when $f$ is strongly convex, PSGM only attains a  sub-linear  convergence rate in general. 

To improve the convergence rate of PSGM, we can use variance reduction as proposed in 
\cite{xiao}. When the objective $f$ has a  finite sum form ($f = n^{-1}\sum_{i=1}^n f_i$), this method computes the full gradient periodically.  Hence, its per iteration cost is dimension dependent. For faster convergence, we can also use the stochastic averaged gradient algorithm (SAGA) in \cite{Roux12}, which requires additional memory. Other modifications do exist to circumvent the convergence speed issue. 

Surprisingly, SGM with constant step-size (SGM-CS) directly attains linear convergence when the strong growth condition (SGC) \cite{Mark13} is satisfied. 
When $f$ has the finite sum structure, SGC can be written as follows with $B>0$:
 \begin{equation}
 \max_{1\leq i\leq n}\|\nabla f_i(x) \|^2 \leq  B \|\nabla f(x)\|^2.
 \end{equation}
Such conditions are also used in \cite{solodov98,Tseng98} for the deterministic incremental gradient method and \cite{Gur15} for the incremental Newton method. Note that 
 \cite{Mark13,solodov98,Tseng98,Gur15} use above condition for $C=\RR^d$.
 
 In this work, we  prove that SGC is also a necessary condition  for the linear convergence of SGM-CS with step-size $\gamma$.  When SGC is violated up to a additive perturbation $\sigma$, we show that PSGM-CS exhibits linear convergence to a noise dominated region, whose distance to the optimal solution is proportional to $\gamma \sigma$. To our knowledge, this result is new. We also derive similar results to the proximal stochastic gradient method.

The paper is organized as follows. We first recall some basic notations in convex analysis in \cite{livre1} below. Section 2 then presents our main results  with a necessary and sufficient condition for the 
 linear convergence of SGM with constant step-size. We also extend these results to the PSGM and  the proximal stochastic gradient method. Section 3 studies the necessary condition in the context of the linear convergence of randomized Kaczmarz algorithm. We conclude in Section 4. 
 
 \noindent{\bf Notations.} 
Given a non empty closed convex set $C$,  the projection of $x$ onto $C$ id denoted by $P_Cx$. The indicator of $C$
is denoted by $\iota_C$.
The proximity operator of a proper lower semicontinuous convex function $g$ is denoted by $\prox_g$. We denote $\dom(g)$ the effective domain of $g$.
The subdifferential of $g$ at $p$ is defined by
 $
 \partial g(p) = \menge{u\in\RR^d}{ (\forall x\in \RR^d)\; g(x) -g(p)\geq \scal{x-p}{u}}.
 $
 When $\partial g$ is a singleton, $g$ is a differentiable function and it is denoted by $\nabla g(p)$. The identity operator is denoted by $\Id$. 
 A single-valued operator $B\colon \RR^d\to\RR^d$ is $\beta$-cocoercive, for some $\beta \in \left]0,+\infty\right[$, if 
 $$(\forall x\in\RR^d)(y\in\RR^d)\;\scal{x-y}{Bx-By}\geq \beta\|Bx-By\|^2.$$
 Given an i.i.d sequence $(\xi_t)_{t\in\NN}$, we denote $\E_{\xi_t}[x]$ is the conditional expectation of $x$ with respect to the history 
 $\xi_{[t-1]} = \{\xi_0,\xi_1,\ldots, \xi_{t-1}\}$.
\section{Main results} 
Let us first recall the proximal stochastic gradient algorithm which was proposed for solving Problem \ref{prob1}.  
 Let $x_0\in\RR^d$ and $(\xi_t)_{t\in\NN}$ be an iid sequence, and let $\gamma_{t} >0$. We iterate as follows
\begin{equation}
\label{algo1}
(\forall t\in\NN)\quad x_{t+1} = \prox_{\gamma_{t} g}(x_t-\gamma_{t} \nabla K(x_t,\xi_t)).
\end{equation}
Let us define the stochastic gradient mapping,
$\mathcal{G}(x_t,\xi_t) = \gamma_{t}^{-1}(x_t- x_{t+1}).$
By the definition of the  proximity operator, there exists $q_{t+1} \in \partial g(x_{t+1})$ such that 
\begin{equation}
\mathcal{G}(x_t,\xi_t)  = q_{t+1} + \nabla K(x_t,\xi_t).
\end{equation}
Our main result can be now stated. 
\begin{theorem} Suppose that the solution set $\mathcal{S}$ is non-empty, and conditioned on $\xi_{[t-1]} = \{\xi_0,\xi_1\ldots,\xi_{t-1}\}$: 
\begin{equation}
\label{e:lin}
(\forall t\in\NN)\quad \E_{\xi_t}[\|x_{t+1} -x^*\|^2] \leq \omega \|x_t-x^*\|^2 +\gamma_{t}^2\sigma^2,
\end{equation}
for some constant $\omega \in \left]0,1\right[$,  constant $\sigma\in\RR$ and $x^*\in\mathcal{S}$. Then, the following holds.
\begin{enumerate}
\item\label{t:1i} We have
\begin{equation}
\label{main1}
\E_{\xi_t}[\|\mathcal{G}(x_t,\xi_t) \|^2] \leq \frac{1}{1-\omega} \|\E_{\xi_{t}}[\mathcal{G}(x_t,\xi_t)] \|^2 + \sigma^2.
\end{equation}
\item \label{t:1ii}  If $g\equiv c$ is a constant function, then $q_t\equiv 0$ and
 \begin{equation}
\label{main2}
\E_{\xi_t}[\|\nabla K(x_t,\xi_t) \|^2] \leq \frac{1}{1-\omega} \|\nabla f(x_t) \|^2 +\sigma^2.
\end{equation}
\end{enumerate}
\end{theorem}
\begin{proof} \ref{t:1i}:
We have $(\forall t\in\NN)\; x_{t+1} =x_t-\gamma_t\mathcal{G}(x_t,\xi_t)$. Hence, we have
\begin{align*}
\|x_{t+1} - x^*\|^2& = \| x_t - x^* - \gamma_{t} \mathcal{G}(x_t,\xi_t) \|^2\notag\\
&=\| x_t - x^* \|^2 -2\gamma_{t} \scal{x_t - x^*}{\mathcal{G}(x_t,\xi_t)} +\gamma_{t}^2 \|\mathcal{G}(x_t,\xi_t) \|^2.
\end{align*}
Since $x_t$ depends on the history $\xi_{[t-1]}$, and independent of $\xi_t$,
taking conditional  expectation with respect to $\xi_{[t-1]}$, we obtain
\begin{alignat}{2}
\label{e:lin2}
\E_{\xi_t}[\|x_{t+1} - x^*\|^2]
&=\| x_t -x^* \|^2 -2 \gamma_{t} \scal{x_t - x^*}{\E_{\xi_{t}}[\mathcal{G}(x_t,\xi_t)]} +\gamma_{t}^2 \E_{\xi_t}[ \|\mathcal{G}(x_t,\xi_t) \|^2].
\end{alignat}
Now, using \eqref{e:lin}, we derive from \eqref{e:lin2} that
\begin{alignat}{2}
\label{e:lin3}
  \gamma_{t}^{2}\E_{\xi_t}[ \|\mathcal{G}(x_t,\xi_t) \|^2]& \leq (\omega-1) \|x_{t} - x^*\|^2 + 2 \gamma_{t} \scal{x_t - x^*}{\E_{\xi_{t}}[\mathcal{G}(x_t,\xi_t)]} +\gamma_{t}^2\sigma^2
  \end{alignat}
 Note that, by  Cauchy-Schwarz inequality,
 \begin{alignat}{2}
  2 \gamma_{t} \scal{x_t - x^*}{\E_{\xi_{t}}[\mathcal{G}(x_t,\xi_t)]} &\leq 2\gamma_t \|x_t - x^* \| \|\E_{\xi_{t}}[\mathcal{G}(x_t,\xi_t)] \|\notag\\
  &\leq (1-\omega)\|x_t - x^*\|^2 + \frac{\gamma_{t}^2}{1-\omega} \|\E_{\xi_{t}}[\mathcal{G}(x_t,\xi_t)] \|^2,
  \end{alignat}
  which implies that 
  \begin{equation}
    2 \gamma_{t} \scal{x_t - x^*}{\E_{\xi_{t}}[\mathcal{G}(x_t,\xi_t)]} +(\omega-1)\|x_t - x^*\|^2 \leq \frac{\gamma_{t}^2}{1-\omega} \|\E_{\xi_{t}}[\mathcal{G}(x_t,\xi_t)] \|^2.
  \end{equation}
  Therefore, it follows from \eqref{e:lin3} that 
  \begin{equation}
    \gamma_{t}^{2}\E_{\xi_t}[ \|\mathcal{G}(x_t,\xi_t) \|^2] \leq  \frac{\gamma_{t}^2}{1-\omega} \|\E_{\xi_{t}}[\mathcal{G}(x_t,\xi_t)] \|^2 +\gamma_{t}^2\sigma^2,
  \end{equation}
  which proves \eqref{main1}.
  
  \ref{t:1ii}. Since $g$ is a constant function, for all $x$, $\partial g(x) = \{0\} $, hence $q_{t+1}=0$ and  $\mathcal{G}(x_t,\xi_t) = \nabla K(x_t,\xi_t)$.
\end{proof}
\begin{remark} In the remainder of this paper, if $\sigma^2 >0$, \eqref{main2} is called the weak growth condition (WGC) of $f$; and if $\sigma^2=0$, \eqref{main2} is called the growth condition (GC) of $f$.
The growth condition is much weaker than the strong growth condition. We have 
\begin{equation}
(SGC)\Longrightarrow (GC) \Longrightarrow (WGC).
\end{equation}
\end{remark}
\begin{remark} Our necessary condition \eqref{main1} remains valid for non-convex, non-smooth $f$. It also holds in the context of solving monotone inclusions 
\cite{LSB14a} where $\nabla f$ is replaced by any cocoercive operator $B$ and $\partial g$ is replaced by any  maximally monotone operator $A$ (see \cite{livre1} for definitions), and 
$\nabla K(x_t,\xi_t)$ is replaced by any stochastic estimate $r(x_t,\xi_t)$  of $Bx_t$ as in \cite{LSB14a}. More precisely, let us consider the following iteration 
\begin{equation}
\label{e:monoiter}
x_{t+1} = (\Id + \gamma_t A)^{-1}(x_t-\gamma_t r(x_t,\xi_t)),
\end{equation}
aiming at solving the following monotone inclusion
\begin{equation}\label{e:mono}
\text{find $x^*\in\RR^d$ such that }\;  0 \in Ax^* + Bx^*.
\end{equation}
 Suppose that the solution set $\mathcal{S}_1$ of \eqref{e:mono} is non-empty, and \eqref{e:lin} is satisfied for some $x^*\in \mathcal{S}_1$. Then \eqref{main1} holds.
 \end{remark}

 In the next theorem, we show that \eqref{main2} 
is also a sufficient condition for linear convergence (with $\sigma=0$) of the  stochastic gradient method
for the class of restricted strongly convex function $f$.  
Restricted strong convexity is much weaker than strong convexity, some examples and properties of restricted strongly convex functions
can be found in \cite{Zhang14}. Note that if $f$ is a strongly convex function, $A$ is a linear mapping, 
then the composite function $f\circ A$ is restricted strongly convex.

\begin{theorem} Suppose that $g=\iota_C$ for some non-empty closed convex set $C$ in $\RR^d$ such that the set $\mathcal{S}$ of solutions is non-empty,
and  that $f$ is $\mu$-restricted strongly convex on $C$ in the sense that $(\forall x\in C)\; f(x) - f(P_{\mathcal{S}}x) \geq 0.5\mu \|x-P_{\mathcal{S}}x\|^2 $. Suppose that  $f^\star = \inf_{x\in\RR^d}f(x) \in\RR$,
and the following weak growth condition is satisfied: 
\begin{equation}
\label{SG}
\E_{\xi_t}[\|\nabla K(x_t,\xi_t)\|^2] \leq M \|\nabla f(x_t)\|^2 +\sigma^2
\end{equation}
for some positive constant $M$ such that $\mu < 4LM$, and $\sigma\in\RR$. Let us define 
$\gamma_t =
\gamma \leq 1/(2LM)
$, and set $\rho = \gamma\mu(1-\gamma LM) \in ]0,1[$. Then, it holds that 
\begin{equation}
\E_{\xi_t}[\|x_{t+1}-\overline{x}_{t+1} \|^2] \leq (1-\rho)\|x_{t}-\overline{x}_t\|^2 +\gamma^2 \sigma_{1}^2,
\end{equation}
where $\overline{x}_t$ is the projection of $x_t$ onto the set of solutions $\mathcal{S}$ and $\sigma^{2}_1 =\sigma^2 + 2LM(\min_{x\in C}f(x)-f^\star) $.
\end{theorem}
\begin{proof} Since $\mathcal{S} \subset C$ and $\overline{x}_t \in C$, we have
\begin{alignat}{2}
\|x_{t+1} -\overline{x}_{t+1}\|^{2} &\leq \|x_{t+1} -\overline{x}_{t}\|^{2}\notag\\
&= \|P_{C}(x_t-\gamma \nabla K(x_t,\xi_t)) - P_C\overline{x}_t\|^2\notag\\
&\leq \|x_{t}-\overline{x}_t -\gamma \nabla K(x_t,\xi_t) \|^2,
\end{alignat}
where the last inequality follows from the non-expansiveness of $P_C$. Hence,  we obtain,
\begin{alignat}{2}
\|x_{t+1} -\overline{x}_{t+1}\|^{2}&\leq \|x_t-\overline{x}_t\|^2 -2\gamma\scal{x_t-\overline{x}_t}{\nabla K(x_t,\xi_t)} +\gamma^2\|\nabla K(x_t,\xi_t)\|^2.
\end{alignat}
Since $x_t$ depends on the history $\xi_{[t-1]}$, and independent of $\xi_t$,
taking conditional  expectation with respect to $\xi_{[t-1]}$, and using the condition \eqref{SG}, we obtain
\begin{alignat}{2}
\E_{\xi_t}[\|x_{t+1} - \overline{x}_{t+1}\|^2]
&\leq\| x_t -\overline{x}_t \|^2 -2 \gamma \scal{x_t - \overline{x}_t}{\nabla f(x_t)} +\gamma^2 M \|\nabla f(x_t) \|^2+\gamma^2\sigma^2.\label{ss1}
\end{alignat}
Using the $L$-Lipschitz continuous of $\nabla f$, it follows that 
\begin{equation}
\|\nabla f(x_t)\|^2\leq 2L (f(x_t)-f^\star) = 2L (f(x_t)-f(\overline{x}_t)) + 2L(f(\overline{x}_t) -f^\star) \label{ss2}.
\end{equation}
Moreover, using the convexity of $f$, we also have 
\begin{equation}
\scal{\overline{x}_t-x_t}{\nabla f(x_t)} \leq f(\overline{x}_t) - f(x_t) \label{ss3}.
\end{equation}
Inserting \eqref{ss2} and \eqref{ss3} into \eqref{ss1}, we get
\begin{alignat}{2}
\E_{\xi_t}[\|x_{t+1} - \overline{x}_{t+1}\|^2]
&\leq\| x_t -\overline{x}_t \|^2 + 2 \gamma( f(\overline{x}_t) - f(x_t))  +\gamma^2 2LM  (f(x_t)-f(\overline{x}_t)) +\gamma^2\sigma_{1}^2\notag\\
&= \| x_t -\overline{x}_t \|^2 - 2\gamma(1-\gamma LM)(f(x_t) -f(\overline{x}_t)) +\gamma^2\sigma_{1}^2\notag\\
&\leq \| x_t -\overline{x}_t \|^2  -\gamma\mu(1-\gamma LM)\| x_t -\overline{x}_t \|^2+\gamma^2\sigma_{1}^2 \notag\\
&= (1-\rho) \| x_t -\overline{x}_t \|^2 +\gamma^2\sigma_{1}^2,
\end{alignat}
where the last inequality follows from the $\mu$-restricted strongly convex of $f$, which proves the desired result.
\end{proof}
\begin{remark} If $f$ is restricted strongly convex, we can find $\mu \leq 4LM$. Hence, 
when $C=\RR^d$, $\sigma=0$ and $\mu \leq 4LM$, the optimal choice of $\gamma$ is $1/(2LM)$. 
\end{remark}
\begin{example}\label{ex:1}
 Suppose that $K(\cdot,\xi)$ is a differentiable function with $L_{\xi}$-Lipschitz gradient such that $L_{0}=\sup_{\xi\in\Omega}L_{\xi} <+\infty$.  
 If $f$ is $\mu$-restricted strongly convex and $(\forall t\in\NN)\;\E_{\xi_t}[\|\nabla K(\overline{x}_t,\xi_t) \|^2] \leq \beta^2  <+\infty$ almost surely, for some positive constant $\beta$,
 then
\begin{equation}
(\forall t\in \NN)\; \E_{\xi_t}[\|\nabla K(x_t,\xi_t)\|^2] \leq (4L_0/\mu)\|\nabla f(x_t) \|^2 +2\E_{\xi_t}[\|\nabla K(\overline{x}_t,\xi_t) \|^2],
\end{equation}
where $\overline{x}_t$ is the projection of $x_t$ onto the set of minimizers $\mathcal{S}$. Hence, the condition \eqref{SG}
is satisfied with $M= 4L_0/\mu$ and $\sigma^2 = 2\beta^2$.
\end{example}
\begin{proof} Indeed, using the cococercivity of $\nabla K(\cdot,\xi)$, we have
\begin{alignat}{2}
\E_{\xi_t}[\|\nabla K(x_t,\xi_t)\|^2]  &\leq 2 \E_{\xi_t}[\|\nabla K(x_t,\xi_t) - \nabla K(\overline{x}_t,\xi_t)\|^2]+ 2\E_{\xi_t}[\|\nabla K(\overline{x}_t,\xi_t) \|^2]\notag\\
 &\leq 2L_0\scal{x_t-\overline{x}_t}{\nabla f(x_t) -\nabla f(\overline{x}_t)}+ 2\E_{\xi_t}[\|\nabla K(\overline{x}_t,\xi_t) \|^2]\notag\\
  &\leq 2L_0\scal{x_t-\overline{x}_t}{\nabla f(x_t)}+ 2\E_{\xi_t}[\|\nabla K(\overline{x}_t,\xi_t) \|^2]\label{asa}.
\end{alignat}
Suppose that $f$ is $\mu$-restricted strongly convex. We have 
$f(x_t) -f(\overline{x}_t) \geq 0.5 \mu \|x_t-\overline{x}_t\|^2$ and $f(\overline{x}_t) -f(x_t) \geq \scal{x_t-\overline{x}_t}{-\nabla f(x_t)}$. Adding them, we get 
$\scal{x_t-\overline{x}_t}{\nabla f(x_t)}\geq 0.5 \mu \|x_t-\overline{x}_t\|^2$. Therefore, $\|x_t-\overline{x}_t\| \leq (2/\mu)\|\nabla f(x_t) \|$. We have
\begin{alignat}{2}
2L_0\scal{x_t-\overline{x}_t}{\nabla f(x_t)}\leq 2L_0 \|x_t-\overline{x}_t\| \|\nabla f(x_t)\|\leq 4L_0\mu^{-1}\|\nabla f(x_t)\|^2.
\end{alignat}
Inserting this into \eqref{asa}, we get the result.
\end{proof}
\begin{example}
Since $\E_{\xi_t}[\scal{\nabla K(x_t,\xi_t) -\nabla f(x_t)}{\nabla f(x_t)}] =0$, we have
\begin{equation}
\E_{\xi_t}[\|\nabla K(x_t,\xi_t)\|^2] = \| \nabla f(x_t)\|^2 + \E_{\xi_t}[\|\nabla K(x_t,\xi_t) -\nabla f(x_t)\|^2].
\end{equation}
Therefore, under the standard condition $\E_{\xi_t}[\|\nabla K(x_t,\xi_t) -\nabla f(x_t)\|^2] \leq \sigma^2$, the condition \eqref{SG} is satisfied.
\end{example}

In the case when $(\forall t\in\NN)\;\partial g(x_t) = \{Q\}$, then $q_{t+1} = -\nabla f(x^*)$. In this case, the necessary condition, with $\sigma=0$, becomes 
$\E_{\xi_t}[\|\nabla K(x_t,\xi_t)-\nabla f(x^*)\|^2] \leq M \|\nabla f(x_t)-\nabla f(x^*)\|^2$. Whenever, this condition is satisfied and $f$ is strongly convex, 
we can prove that the linear convergence of the proximal stochastic gradient method is obtained. However, 
the following result shows that \eqref{SG} is also a sufficient for linear convergence to a noise dominated region of the proximal stochastic gradient method.

\begin{proposition} \label{p:1}
Suppose that $f$ is $\mu$-strongly convex, and the weak growth condition \eqref{SG} is satisfied. Set
$\sigma_{1}^2=2(1+2M) \|\nabla f(x^*)\|^2 +  2\sigma^2$, 
where $x^*$ is the optimal solution. Let  $\gamma_t =\gamma$ be chosen such that $\rho = \gamma\mu(1-2\gamma LM) \in ]0,1[$. 
Then, for iteration \eqref{algo1}, we have 
\begin{equation}\label{111}
\E_{\xi_t}[\|x_{t+1}-x^* \|^2] \leq (1-\rho)\|x_{t}-x^*\|^2 +\gamma^2\sigma_{1}^2.
\end{equation}
\end{proposition}
\begin{proof} Since \eqref{SG} is satisfied. Then
\begin{alignat}{2}
\E_{\xi_t}[\|\nabla K(x_t,\xi_t)-\nabla f(x^*)\|^2] &\leq 2M\|\nabla f(x_t)\|^2 + 2\|\nabla f(x^*)\|^2+2\sigma^2\notag\\
&\leq 4M\|\nabla f(x_t) -\nabla f(x^*)\|^2 + 2(1+2M) \|\nabla f(x^*)\|^2 +  2\sigma^2.\label{SG1}
\end{alignat}
Since $\prox_{\gamma g}$ is non-expansive and $x^* =\prox_{\gamma g}(x^*-\gamma \nabla f(x^*))$,
 we have 
\begin{alignat}{2}
\|x_{t+1} -x^*\|^2&\leq \|x_t-x^*-\gamma(\nabla K(x_t,\xi_t) -\nabla f(x^*)) \|^2\notag\\
&= \|x_t-x^*\|^2 -2\gamma \scal{x_t-x^*}{\nabla K(x_t,\xi_t) -\nabla f(x^*)} +\gamma^2\|\nabla K(x_t,\xi_t) -\nabla f(x^*) \|^2.\notag
\end{alignat}
Taking conditional expectation both sides and using  \eqref{SG1}, we get 
\begin{alignat}{2}
\E_{\xi_t}[\|x_{t+1} -x^*\|^2]& \leq \|x_t-x^*\|^2 -2\gamma \scal{x_t-x^*}{\nabla f(x_t) -\nabla f(x^*)} \notag\\
&\quad +\gamma^24M\|\nabla f(x_t) -\nabla f(x^*) \|^2+\gamma^2\sigma_{1}^2\notag\\
&\leq \|x_t-x^*\|^2 -(2\gamma -\gamma^2 4LM)\scal{x_t-x^*}{\nabla f(x_t) -\nabla f(x^*)} +\gamma^2\sigma_{1}^2\notag\\
&\leq (1-\rho)\|x_t-x^*\|^2 +\gamma^2\sigma_{1}^2,
\end{alignat}
where the first inequality follows from the cocoercivity of $\nabla f$, and the last equality follows from the strong convexity of $f$. 
\end{proof}
\begin{remark} When $\sigma_1 > 0$, \eqref{111} implies that we get linear converge to
a noise dominated region proportional to $\gamma\sigma_1$. In the case, $g=0$ and $g=\iota_C$, this kind of convergence result can be found in \cite{Need}
and \cite{Nedic}, respectively. For the case of the stochastic proximal point algorithm, it is presented in \cite{Ryu16}.
\end{remark}
\begin{remark} The proposition above remains  valid for \eqref{e:monoiter}. Here $\nabla f$ and $\partial g$ are replaced by a cocoercive, strongly monotone operator $B$ and a maximally 
monotone operator $A$, respectively; and $\nabla K(x_t,\xi_t)$ is replaced by unbiased estimate  $r(x_t,\xi_t)$ of $Bx_t$ as in \cite{LSB14a}.
\end{remark}
\begin{remark} Under the same conditions as in Proposition \ref{p:1},  we see that in the case when $\gamma_t$ is not constant, $\gamma_t=\mathcal{O}(1/(1+t))$, then there exists $t_0\in\NN$
such that 
\begin{equation}
(\forall t\geq t_0)\quad\E[\|x_t-x^*\|^2] = \mathcal{O}(1/t),
\end{equation} 
where the expectation is taken over the whole history. This convergence rate is known in \cite{LSB14a}.
\end{remark}
\section{Special instances of the necessary condition}
We have already proved that the growth condition 
\begin{equation}
\label{SG0}
\E_{\xi_t}[\|\nabla K(x_t,\xi_t)\|^2] \leq M \|\nabla f(x_t)\|^2,
\end{equation}
is the necessary and sufficient condition for linear convergence of the stochastic gradient method for the class of
convex differentiable function with gradient Lipschitz and restricted strongly convex. 
We study this necessary condition to establish the linear convergence of randomized  Kaczmarz algorithm \cite{Stroh} and 
of the stochastic gradient method as in \cite{Mark13}.

Let $(a_i)_{1\leq i\leq m}$ be sequence of colum vectors, with norm $1$,  in $\RR^d$ and $b\in \RR^m$  with $(m\geq d)$. 
Set $(\forall i \in \{1,\ldots,m\})\; C_i = \menge{x\in\RR^d}{\scal{a_i}{x} =b_i}$. Let $A$ be a matrix with rows $(a^{T}_i)_{1\leq i\leq m}$.
Let us consider the problem
\begin{equation}
\underset{x\in \RR^d}{\text{minimize}} \; f(x)= \frac{1}{2m}\sum_{i=1}^m\|x-P_{C_i}x\|^2,
\end{equation}
under the assumptions that $\emp\not=\cap_{i=1}^mC_i$ and $A$ is a full rank matrix. 
Set $f_i = 0.5\|x-P_{C_i}x\|^2$. Let $i_k$ be chosen uniformly at random in $\{1,\ldots,m\}$. 
Then 
\begin{alignat}{2}
\E_{i_k}[\|\nabla f_{i_k}(x)\|^2] &=\frac{1}{m}\sum_{i=1}^m\|x-P_{C_i}x\|^2\notag\\
&= \frac{1}{m}\sum_{i=1}^m|\scal{a_i}{x}-b_i |^2\notag\\
&= \frac{1}{m}\|Ax-b\|^2.
\end{alignat}
Let us define $A^{\dag} = (A^TA)^{-1}A^T$. Then $\|A^{\dag}(Ax-b)\| \leq \|(A^TA)^{-1}\|\|A^T(Ax-b)\|$. Let $x^*\in \cap_{i=1}^mC_i$. Then
$\|A^{\dag}(Ax-b)\| = \|x- A^{\dag}b\| = \|x-x^*\| \geq \|A\|^{-1}\|A(x-x^*)\|=\|A\|^{-1}\|Ax-b\|$. Therefore, upon setting $M=m\|A\|^2\|(A^TA)^{-1}\|^2 $, we have
\begin{alignat}{2}
\E_{i_k}[\|\nabla f_{i_k}(x)\|^2] &= \frac{1}{m}\|Ax-b\|^2 \leq \frac{\|A\|^2\|(A^TA)^{-1}\|^2 }{m} \|A^T(Ax-b)\|^2
= M \|\nabla f(x)\|^2,
\end{alignat}
which shows that the necessary condition \eqref{main2} is satisfied with $\sigma=0$. Furthermore, since the objective function is restricted strongly convex, 
in view of above theorem, the stochastic gradient method converges linearly which was also known in \cite{Stroh} with $\gamma=1$.
Further connections to the randomized Kaczmarz algorithm can be found in  \cite{Need} where the case $\cap_{i=1}^mC_i=\emp$ is investigated. In this work, 
they show that the stochastic gradient method converges linearly to
a noise dominated region proportional to $\gamma\sigma$ with
 $\sigma = 2\E_{i_k}[\|\nabla f_{i_k}(x^*)\|^2]$.

In the general case of $f_i$.  The condition \eqref{SG0} is satisfied when 
\begin{equation}
\label{SG2}
(\forall i\in\{1,\ldots, n\}) (\forall x\in C)\quad \|\nabla f_i(x)\|^2 \leq M \|\nabla f(x)\|^2.
\end{equation}

\section{Conclusions}
The strong growth condition is used in \cite{solodov98} where the incremental gradient method converges with a sufficiently small constant step size and 
in \cite{Tseng98} where incremental gradient method converges linearly with  a sufficiently small constant step size. Furthermore,
 and it is also recently used in \cite{Gur15} for linear convergence of the incremental Newton method,
and in \cite{Mark13} for linear convergence of the stochastic gradient method. All the existing work agrees that the strong growth condition 
is very strong, it requires at least the vanishing of stochastic gradient  at optimal solution. Unfortunately, our work shows that 
it is necessary to achieve linear convergence.

\noindent {\bf Acknowledgments.} The authors would like to thank Yen-Huan-Li, Ahmet Alacaoglu (they are Ph.d students of LIONS-EPFL) for  useful discussions.
The work of B. Cong Vu and V. Cevher was supported by
European Research Council (ERC) under the European Union's Horizon 2020 research and innovation
programme (grant agreement no 725594 - time-data).

\end{document}